\documentclass[a4paper,12pt]{article}
\usepackage{hyperref}

\usepackage{mathrsfs}
\usepackage{ifthen}
\usepackage{verbatim}
\usepackage[english]{babel}
\usepackage{fancyhdr}
\usepackage{enumitem}
\usepackage{amsmath,amssymb,amsthm,graphicx}
\usepackage{latexsym}
\usepackage{ifpdf}
\usepackage{appendix}
\usepackage{color}
\usepackage{bbm}

\usepackage{tikz}
\usetikzlibrary{matrix,arrows,decorations.pathmorphing}

\newtheorem{Theorem}{Theorem}[section]
\newtheorem{Lemma}[Theorem]{Lemma}

\newtheorem{Proposition}[Theorem]{Proposition}

\newtheorem{Definition}[Theorem]{Definition}

\theoremstyle{definition}
\newtheorem{Remark}[Theorem]{Remark}

\def\l{\left}
\def\r{\right}
\newcommand{\seq}[1]{\l\{#1\r\}}
\newcommand{\bra}[1]{\l(#1\r)}

\newcommand{\abs}[1]{\l|#1\r|}

\def\norm#1{\left \Vert #1 \right \Vert}

\renewcommand{\hat}{\widehat}

\def\gap{\; \; \;}

\newcommand{\PP}{\mathbb{P}}
\newcommand{\EE}[1]{\mathbb{E}\ex{#1}}

\newcommand{\Ee}[1]{{\mathbb{E}}\ex{#1}}

\newcommand{\RR}{\mathbb{R}}

\newcommand{\NN}{\mathbb{N}}
\newcommand{\Gg}{\mathcal{G}}

\def\id{\textbf{\textup{1}}}

\newcommand{\Ll}{\mathcal{L}}

\newcommand{\Mm}{\mathcal{M}}

\newcommand{\Aa}{\mathcal{A}}

\def\EE{\mathbb{E}}
\def\Ww{\mathcal{W}}
\def\Ll{\mathcal{L}}

\def\Mm{\mathcal{M}}

\def\Aa{\mathcal{A}}

\newcommand{\Var}[1]{\textup{Var}({#1})}
\usepackage[us]{datetime}
\usepackage{geometry}
\def\Ee{\mathcal{E}}

\numberwithin{equation}{section}
\numberwithin{figure}{section}

\begin{document}
\title{Wasserstein approximations of the L\'{e}vy area random walk via polynomial perturbations of Gaussian distributions}
\author{Guy Flint\\ \\ \small{\textit{ Mathematical Institute, University of Oxford, Woodstock Road, OX2 6GG, England}}}
\date{\today}
\maketitle

\begin{changemargin}{0.3cm}{0.3cm} 
\abstract{
We construct a coupling between the random walk composed of L\'{e}vy area increments from a $d$-dimensional Brownian motion and a random walk composed of quadratic polynomials of Gaussian random variables. This coupling construction is used to produce a new pathwise approximation scheme for stochastic differential equations in the preprint \cite{flintCoupling}. The coupling arguments of the present paper are based extensively on the recent coupling results of Davie in \cite{davie2014kmt,davie2014pathwise,davie2015poly} concerning a multidimensional variant of the Koml\'{o}s-Major-Tusn\'{a}dy theorem and Wasserstein estimates for polynomial perturbations of Gaussian measures.
}
\end{changemargin}


\smallskip
\noindent \texttt{Guy.Flint@maths.ox.ac.uk}

\section{Introduction}

Let $W : t \in[0,1]\mapsto (W_1(t), \ldots, W_d(t)) \in  \RR^d$ be a standard Brownian motion and introduce notation for the Brownian and L\'{e}vy area increments as follows: given $N\in\NN$ (and $h=N^{-1}$), we set
\begin{align}
W^{(j)}_k :&= W_k(jh,(j+1)h) := W_k((j+1)h)-W_k(jh);\notag\\
A^{(j)}_{kl} :&= \frac{1}{2}\int^{(j+1)h}_{jh} \seq{W_k(jh,t)\, dW_l(t) - W_l(jh,t)\, dW_k(t)} = A_{kl}(jh,(j+1)h)\label{e-original-increment}
\end{align}
and write $A^{(j)} = \{A^{(j)}_{kl}\}_{1\leq k < l \leq d} \in \RR^{\frac{d}{2}(d-1)}$. 
Simulation of L\'{e}vy area increments is important for the numerical approximation of stochastic differential equations (for example, the $N$-step Milstein scheme \cite{mil1974approximate} requires $N$ L\'{e}vy area increments in order to achieve a strong approximation error of order $O(N^{-1})$). 
In the case of $d=1$ it is a trivial exercise to generate iterated Brownian integrals but for $d\geq 2$, this task becomes a hard problem. In the case of $d=2$ efficient algorithms based on Fourier expansions are available for generating double integrals (that is, L\'{e}vy area increments $A^{(j)}$)  but they involve a significant computational cost  \cite{gaines1994random,ryden2001simulation,wiktorsson2001joint}. The general case of $d>2$ is still out of reach. The different coordinates of L\'{e}vy area are uncorrelated but not independent (as we will show in Lemma \ref{l-area-decomposition} below), which makes joint simulation extremely difficult. 

As an approximation of L\'{e}vy area, in the recent preprint \cite{davie2014kmt} Davie substitutes each $A^{(j)}$ with a suitable quadratic polynomial of Gaussian random variables denoted by $B^{(j)}$. The construction is such that $A^{(j)}$ and $B^{(j)}$ share the same mean and covariance structure along with the same underlying Brownian increment $W^{(j)}$. To introduce $B^{(j)}$ we must first closely examine the original Brownian L\'{e}vy area increments. The following lemma gives a simple decomposition of $A^{(j)}$ into parts dependent and independent of the corresponding Brownian increment $W^{(j)}$. 

\begin{Lemma}\label{l-area-decomposition}
For all $1\leq k < l \leq d$:
\begin{equation}\label{e-area-decomposition}
A^{(j)}_{kl} = \zeta^{(j)}_k W^{(j)}_l - \zeta^{(j)}_l W^{(j)}_k + K^{(j)}_{kl},
\end{equation}
where the $\zeta^{(j)}_k$: $k=1,\ldots,d$, $K^{(j)}_{kl}$: $1\leq k < l \leq d$, are mutually uncorrelated (but not independent), independent of $W^{(j)}$ and have mean zero. Moreover,  
\[
\Var{\zeta^{(j)}_k} = \frac{1}{12N} \gap\text{ and }\gap \Var{K^{(j)}_{kl}}=\frac{1}{12N^2}.
\] 
\end{Lemma}
\begin{proof}
We suppose $j=0$ for simplicity and begin by decomposing the $A^{(0)}$ increment into parts dependent and independent of the Brownian increment $W(h)$. To this end, following \cite[\S7]{davie2014pathwise} we can write
\[
W_k(t)=h^{1/2}B_k(t/h)+th^{-1/2}Z_k \gap t\in [0,h],
\] 
where $B_1,\ldots,B_d$ are independent standard Brownian bridges on $[0,1]$ and $Z_k=h^{-1/2}W_k(h)$ are independent $N(0,1)$ random variables (which are independent of the $B_j$). Also write $B_0(t)=t$ and set 
\[
K_\alpha := \int^1_0 \int^{t_l}_0 \ldots \int^{t_2}_0 dB_{j_1}(t_1)\ldots dB_{j_l}(t_l)
\] 
for an index $\alpha=(j_1,\ldots,j_l) \in \{0,1\ldots,d\}^l$. For such an index it can be shown that (see \cite[\S7]{davie2014pathwise})
\[
I_\alpha := \int^h_0 \int^{t_l}_0 \ldots \int^{t_2}_0 dW_{j_1}(t_1)\ldots dW_{j_l}(t_l) = h^{(l(\alpha)+n(\alpha))/2}\sum_{\beta=(i_1,\ldots,i_l)} K_\beta \prod_{k:i_k<j_k} Z_{j_k},
\]
where the sum is over all $\beta=(i_1,\ldots,i_l)$ such that for each $k\in\seq{1,\ldots,l}$ we have either $i_k=j_k$ or $i_k=0<j_k$. Here we have used $l(\alpha)$ and $n(\alpha)$ to denote the length and number of zero entries of $\alpha$ respectively. 
Noting the antisymmetry $K_{kl}=-K_{lk}$ for $0\leq k < l$, it follows that 
\begin{align*}
A^{(0)}_{12} 
&= \frac{1}{2}\bra{I_{12}-I_{21}} = h\bra{K_{10}Z_2 - K_{20}Z_1+K_{12}},
\end{align*}
where 
\[
K_{12} = \int^1_0 B_1(t)\, dB_2(t) \textup{ and } K_{j0} = \int^1_0 B_j(t)\, dt \textup{ for } j\in\seq{1,2}. 
\]
Thus $\zeta^{(0)}_j := h^{1/2}K^{(0)}_{j0}$ for $j=1,2$, and $K^{(0)}_{12} := hK_{12}$ gives the claimed decomposition. 
The variances follow from It\^{o}'s isometry. For the calculations we refer to Lemma 7 of \cite{davie2014pathwise}. 
\end{proof}

The fact that for fixed $j$, the increments $\zeta^{(j)}$ and $K^{(j)}$ are not independent makes them (and consequently $A^{(j)}$) very difficult to simulate numerically. A natural solution would be to approximate these two variables with normal random variables $z^{(j)}$, $\lambda^{(j)}$ with the correct mean and moments, to produce a Gaussian approximation $B^{(j)}$ for $A^{(j)}$. Since uncorrelated Gaussian random variables are necessarily independent, simulation is much easier. This is precisely what Davie proposes in \cite{davie2014kmt}; define the following independent normal random variables
\begin{align*}
W^{(j)} \sim N\bra{0,\frac{1}{N}I_d}, \,\,\,\, z^{(j)} &\sim N\bra{0,\frac{1}{12N}I_d},\notag\\ 
\lambda^{(j)} = (\lambda^{(j)}_{kl})_{1\leq k < l \leq d} &\sim N\Big(0,\frac{1}{12N^2}I_{\frac{d(d-1)}{2}}\Big),\label{e-normal}
\end{align*}
for each $j \in \{0,1,\ldots,N-1\}$. Then set $\{B^{(j)}\}_{1\leq k < l \leq d}$ to be quadratic polynomial
\begin{equation}\label{e-b-def}
B^{(j)}_{kl} := z^{(j)}_k W_l^{(j)} - z_l^{(j)}W_k^{(j)} + \lambda_{kl}^{(j)}. 
\end{equation}

The main theorem of this paper constructs a coupling of the two random walks with respective increments $A^{(j)}$ and $B^{(j)}$ conditional on sharing the same underlying Brownian increments $W^{(j)}$. 

\begin{Theorem}\label{t-coupling}
For every $p\in [1,\infty)$, there exists a constant $C_p>0$ such that the following holds: given $N\in\NN$, we can construct a coupling between the i.i.d.~sequence $\{A^{(j)}\}_{j=0}^{N-1}$ and i.i.d.~$\{B^{(j)}\}_{j=0}^{N-1}$ random variables, with each pair $(A^{(j)}, B^{(j)})$ defined by (\ref{l-area-decomposition}) and (\ref{e-b-def}) using a common Brownian increment $W^{(j)} \sim N(0,N^{-1}I_d)$, such that 
\[
\max_{r=1,\ldots,N} \abs{\norm{ \sum^{r-1}_{j=0} \bra{A^{(j)}-B^{(j)}}}_{\RR^{\frac{d}{2}(d-1)}}}_{L^p} \leq C_p \frac{\log N}{N}. 
\]
\end{Theorem}

The fact that the random walks share the same Brownian increments is exploited in the paper \cite{flintCoupling} in order to reduce the complexity of the computations involved with the iterated Baker-Campbell-Hausdorff formula.

\subsection*{Outline of the paper}

We begin in the next section by presenting the Wasserstein metric from optimal transport theory. We then give a short overview of polynomial perturbations of Gaussian measures. Utilising the main results of the recent preprint \cite{davie2015poly}, we then prove an extension of \cite[Corollary 3]{davie2014kmt} from quartics to higher order polynomial perturbations. The final section of the paper then uses this extension together with a central limit theorem to generalise the original coupling proof of \cite[Theorem 1]{davie2014kmt} from a bound in the $2$nd Wasserstein metric $\Ww_2$ to $\Ww_p$ for general $p\in [1,\infty)$, thereby establishing Theorem \ref{t-coupling}. 

\section{Coupling and Wasserstein metrics}

We present a brief primer on the Wasserstein metric and couplings. For more details we refer to Villani's surveys in \cite{cedric2003topics,villani2008optimal}.
Since we only work with probability measures on Euclidean space throughout the paper, we restrict our study of the Wasserstein metric to the space of probability measures on $\RR^d$, denoted by $\mathcal{P}(\RR^d)$. The theory can be extended to general Polish spaces.

Given $p\in [1,\infty)$ let $\mathcal{P}_p(\RR^d) \subset \mathcal{P}(\RR^d)$ denote the subspace of measures with finite $p$th moment:
\[
\mathcal{P}_p(\RR^d) := \seq{\mu\in \mathcal{P}(\RR^d) : \int_{\RR^d} \abs{x}^p \, \mu(dx) < \infty}.
\]

\begin{Definition}
Fix $p\in [1,\infty)$. Given $\mu,\nu \in \mathcal{P}_p(\RR^d)$, define the set of transport plans of $\mu$ to $\nu$ by 
\[
\Mm(\mu,\nu) := \seq{ \Phi : \RR^d \to \RR^d \mid \Phi \textup{ measurable and } \Phi_*(\mu) = \nu}.
\]
Here, $\Phi_*(\mu)$ denotes the pushforward measure of $\mu$ under $\Phi$: 
\[
\Phi_*(\mu)(A) = \mu\bra{\Phi^{-1}\bra{A}} \textup{ for all Borel sets }A\subset \RR^d.
\] 
The $p$th Wasserstein distance on $\mathcal{P}_p(\RR^d)$ (using $\rho$) is defined as
\begin{equation}\label{e-def-was}
\Ww_p\bra{\mu,\nu} := \bra{ \inf_{\Phi \in \Mm(\mu,\nu)} \int_{\RR^d} \abs{x-\Phi(x)}_\infty^p\, \mu(dx)}^{1/p}.
\end{equation}
\end{Definition}

\begin{Remark}
Note the Wasserstein metric can be defined using any metric on $\RR^d$. We have chosen the uniform metric $\abs{x-y}_\infty := \max_{i=1,\ldots,d} \abs{x_i-y_i}$ because it possesses the nice property that the size of the vector  $z=(1,\ldots,1) \in \RR^{d}$ does not grow as $d\to\infty$.
\end{Remark}

Certainly we have $\Ww_q(\mu,\nu) \leq \Ww_p(\mu,\nu)$ for all $1\leq q \leq p <\infty$. This is analogous to the relation $L^q \subseteq L^p$ for Lebesgue $L^p$-spaces.

It can be shown that $\Ww_p$ is a genuine metric on $\mathcal{P}_p(\RR^d)$ (\cite[Theorem 7.3]{cedric2003topics}). Moreover, the infimum (\ref{e-def-was}) is actually a minimum under mild regularity conditions on the measures $\mu,\nu$ (for instance, a sufficient condition is that one of the measures is absolutely continuous with respect to Lebesgue measure).
The minimizing transport plan $\hat{\Phi}$ is unique but it is difficult to explicitly find $\hat{\Phi}$ because $\Mm (\mu,\nu)$ possesses no convexity or linear structure in general. However, in the case of $d=1$ classical optimal transport theory reduces the Wasserstein distance to an elegant formulation. While we will not use this result in this paper, its inclusion may be useful for general orientation. 

\begin{Proposition}\cite[Theorem 3.1.2]{rachev1998mass}
Let $F$ and $G$ be distribution functions on $\RR$ corresponding to probability measures $\mu$ and $\nu$ respectively. Suppose that $F$ is continuous with density $f$ and use $G^{-1}$ to denote the generalised inverse of $G$. Then for all $p\in [1,\infty)$,  
\[
\Ww_p\bra{\mu,\nu} = \bra{ \int_{\RR} \abs{G^{-1} \circ {F(x)} - x}_\infty^p f(x)\, dx}^{1/p}.
\]
In particular, for each $p$ the minimizing transport plan $\hat{\Phi}\in\Mm(\mu,\nu)$ is given by $\hat{\Phi}:=G^{-1}\circ F$.
\end{Proposition}

An equivalent definition of
$\Ww_p$ is as follows. 

\begin{Proposition}\cite[Theorem 9.4]{villani2008optimal}\label{p-was-eq}
Fix $p\in [1,\infty)$ and let $\mu,\nu \in \mathcal{P}_p(\RR^d)$ with one of the measures being absolutely continuous with respect to Lebesgue measure. Then 
\[
\Ww_p\bra{\mu,\nu} = \inf \seq { \abs{X-Y}_{L^p} :  \Ll(X) = \mu, \Ll(Y) = \nu}. 
\]
Specifically, the infimum is taken over all distributions of the $\RR^d$-valued random variables $X$ and $Y$ with marginal distributions $\mu$ and $\nu$ respectively.
\end{Proposition}

This proposition inspires the following definition.

\begin{Definition}[Coupling]
A pair of random variables $(X,Y)$ with the correct marginals $\mu$, $\nu$ is known as a coupling of the two probability measures. 
\end{Definition}

Thus the Wasserstein distance is given by the $L^p$-distance  between the optimal coupling. That is, if $\Ll(X)=\mu$, $\Ll(Y)=\nu$ then certainly $\Ww_p(\mu,\nu) \leq \abs{X-Y}_{L^p}$. 

\begin{Remark}
The topology induced by $\Ww_p$ on $\mathcal{P}_p(\RR^d)$ is slightly stronger than the weak topology; namely, convergence of a sequence $\{\mu_n\}_{n=1}^\infty \subset \mathcal{P}_p(\RR^d)$ to a measure $\mu \in \mathcal{P}_p(\RR^d)$ in $\Ww_p$ is equivalent to weak convergence plus a uniform bound on the $p$th moments of the measures $\{\mu_n\}_{n=1}^\infty$ (\cite[Theorem 6.9]{villani2008optimal}). In symbols, 
\[
\Ww_p\bra{\mu_n,\mu} \to 0 \iff \mu_n \to \mu \textup{ weakly and } \sup_n \int_{\RR^d} \abs{x}^p\, \mu_n(dx) < \infty.
\]
\end{Remark}

\begin{Remark}
The metric $\Ww_p$ originates from the {Monge-Kantorovich mass transportation problem}, first introduced by Monge in 1781 \cite{monge1781memoire}, then rediscovered many times in various forms since by Kantorovich \cite{kantorovich_russian}, L\'{e}vy among others. 
The first modern definition was given by the algebraic $K$-theorist Vaserstein in his sole paper in probability theory \cite{vaserstein}. Vaserstein, (anglicized as \textit{Wasserstein} from the Cyrillic alphabet), used the letter $\mathcal{K}$ for the Wasserstein metric in honour of Kantorovich's original contribution of \cite{kantorovich_russian}. (Amusingly, by total coincidence Kantorovich's work conversely used the letter $\mathcal{W}$). Throughout this paper we have used $\mathcal{W}$ and \textit{Wasserstein}, in agreement with the modern literature. For more historical details we refer to \cite{malrieu2003convergence,villani2008optimal} and \cite[\S12]{davie2014pathwise}. 
\end{Remark}

We conclude the section with a useful lemma \cite[Proposition 7.10]{cedric2003topics}.
Given two measures $\mu,\nu$ on $\RR^d$ with respective densities $f,g$ with respect to Lebesgue measure, we use the notation: $\abs{\mu-\nu}(x) := \abs{f(x)-g(x)}$. 

\begin{Lemma}\label{l-useless-lemma}
Let $\mu$ and $\nu$ be two probability measures on $\RR^d$. For $p\in [1,\infty)$, 
\[
\Ww_p(\mu,\nu) \leq 2^{1- 1/p}\seq{\int_{\RR^d} \abs{x}^p\, d\abs{\mu-\nu}(dx)}^{1/p}.
\]
\end{Lemma}

As mentioned in \cite{davie2014pathwise}, this is quite a good bound for $p=1$ but less good for larger $p$. We will use this lemma in the subsequent sections in order to establish $\Ww_p$-estimates.

\subsection*{Notation}

Before proceeding further we establish some notation.\ 

\textbf{Constants.} Throughout the paper, $C,c,\ldots$ denote various deterministic constants (that may vary from line to line). Constants which are dependent upon a variable will have the dependency explicitly stated; for example, $C_p$ denotes a constant dependent on $p$. If a constant $C$ has many dependencies $\alpha,\beta,\gamma, \ldots,p,q$, we will simply write $C=C(\alpha,\beta,\gamma,\ldots,p,q)$. If we are working in the Euclidean space $\RR^d$ we will always ignore dependencies on the dimension $d$ (except for Section \ref{s-poly-perturb}, where we repeat results in \cite{davie2015poly}).\

\textbf{Gaussian measure.} Given a covariance matrix $\Sigma \in \RR^{d\times d}$, let $\phi_\Sigma$ denote the density function of $N(0,\Sigma)$:
\[
\phi_\Sigma(x) = \frac{1}{\bra{2\pi}^{d/2} \bra{\textup{det}\,{\Sigma}}^{1/2}} \exp\bra{-\frac{1}{2}x^t \Sigma^{-1} x}. 
\]
In the case of $\Sigma=I_d$ (the identity matrix) we simply write $\phi(x)=\phi_{I_d}(x)$. 
Given an arbitrary probability measure $\mu$ on $\RR^d$, we will commit a slight abuse of notation by using $\Ww_p(\mu,\phi_\Sigma)$ to denote the Wasserstein distance $\Ww_p(\mu, \Ll\bra{Z})$ where $Z\sim N(0,\Sigma)$.\ 

\textbf{Polynomial spaces.} Let $P$ denote the space of all real-valued polynomials on $\RR^d$ (which we will write as $P(\RR^d)$ when we want to specify the dimension $d$) and define the subspace 
\[
P_\Sigma := \seq{p\in P : \int_{\RR^d} p(x)\phi_\Sigma(x)\, dx = 0}.
\]
Let $P^d$ denote the space of $\RR^d$-valued polynomial functions on $\RR^d$. To be precise, $p\in P^d$ means $p(x)=(p^1(x), \ldots,p^d(x))$ for some $p^j \in P$. Given a polynomial $p^j$,  let $\textup{deg}(p^j)$ denote the highest degree of its terms. Similarly, for a polynomial function $p=(p^1,\ldots,p^d) \in P^d$, set $\textup{deg}(p):= \max_{j=1,\ldots,d} \textup{deg}(p^j)$. 
Lastly, given the dimension $d$, we set $d_1:=\frac{d}{2}(d+1)$ and $d_2:= \frac{d}{2}(d-1)$ to avoid cumbersome sub/superscript notation. 

\section{Polynomial perturbations of Gaussian distributions}\label{s-poly-perturb}

In this section we consider signed measures on $\RR^d$ with a density given by a polynomial perturbation of the standard Gaussian distribution $\phi(x) = (2\pi)^{-d/2}e^{-\abs{x}^2 / 2}$. Our main aim is to prove that if such a signed measure is close to a probability measure $\mu$ (in the form of an estimate similar to that of Lemma 
\ref{l-useless-lemma}), then we can expect the distance $\Ww_p(\mu,\phi)$ to be bounded by the magnitude of the perturbation. This is the content of the following proposition. 

\begin{Proposition}\label{p-poly-perturb-bound}
Fix $n\in\NN$ and $\varepsilon \in (0,1)$. Suppose $\{S_j\}_{j=1}^n \subset P$ is a sequence of polynomials with the absolute values of its coefficients bounded by a universal constant $M>0$ and set $s:= \max_{j=1,\ldots,n} \textup{deg}(S_j)$.  Fix some integer $1\leq n_0\leq n$ and let $\nu_{\varepsilon,n}$ denote the signed measure on $\RR^d$ with density 
\[
\phi(y)\bra{1+\sum_{j=n_0}^n \varepsilon^j S_j(y)}. 
\]
If $\mu$ is a probability measure on $\RR^d$ such that 
\begin{equation}\label{e-hyp-estimate}
\int_{\RR^d} \bra{1+\abs{y}}^p d\abs{\mu - \nu_{\varepsilon,n}}(y)\, \leq \delta, 
\end{equation}
then for all $p\in [1,\infty)$ we have
\[
\Ww_p\bra{\mu, \phi} \leq C_{d,M,n,p,s}\bra{\varepsilon^{n_0} + {\delta}^{1/p} +  \varepsilon^{\frac{n+1}{p}}}. 
\]
\end{Proposition}

The proposition is a strengthened version of \cite[Corollary 3]{davie2014kmt}. Instead of dealing with only quartic perturbations, Proposition \ref{p-poly-perturb-bound} can handle polynomials of arbitrarily high order. 
The proof was suggested to the author by Professor Davie in private communications and the argument relies upon the results in his preprint \cite{davie2015poly}. Thus before presenting its proof we summarise the work of the latter paper.

\begin{Remark}
To be precise, the original polynomial perturbation result of Davie actually considered quartic perturbations of $\phi_\Sigma$ for some arbitrary covariance matrix $\Sigma \in \RR^{d\times d}$. One may ask whether Proposition \ref{p-poly-perturb-bound} can also be generalised to the case of arbitrary $\phi_\Sigma$ rather than just $\Sigma=I_d$. The author conjectures that this is possible since the technical tools from \cite{davie2015poly} remain valid. However, the precise dependencies of the norms $\norm{\Sigma}$ and $\norm{\Sigma}^{-1}$ in the bound of $\Ww_p(\mu,\phi_\Sigma)$ are complicated and it is a non-trivial task to keep track of these quantities throughout the proof. 
Fortunately we only need the case of $\Sigma=I_d$ for the coupling arguments of this paper. 
\end{Remark}  

We now summarise the main contributions of \cite{davie2015poly} and begin by characterising the subspace $P_\Sigma \subset P$ as follows. 

\begin{Lemma}
Define the linear mapping $\Ll_\Sigma : P^d \to P$ by $\Ll_\Sigma p(x) = \nabla \cdot p(x) - x^t\Sigma^{-1} p(x)$. 
The space $P_\Sigma$ is precisely the range of $\Ll_\Sigma$. Moreover, any element  $p\in P_\Sigma$ can be expressed as $\Ll_\Sigma \nabla u$ for some $u\in P$ (and the $u$ is unique up to an additive constant).
\end{Lemma}

The proof is by induction and the divergence theorem (see \cite[Lemma 1]{davie2014pathwise}). 
One consequence of the lemma is that if we define $P^d_G$ to be the set of $p\in P^d$ of the form $p=\nabla u$ with $u\in P$, we have that $\Ll_\Sigma$ is bijective from $P^d_G\to P_\Sigma$ and we can define the inverse linear mapping $\Ll_\Sigma^{-1} : P_\Sigma \to P^d_G$. 

Next, suppose we have a sequence of polynomial functions $\{p_j\}_{j=1}^n \subset P^d_G$.
For each $\varepsilon>0$, we define the polynomial perturbation mapping $\rho_{\varepsilon} :\RR^d\to\RR^d$ by 
\begin{equation}\label{e-pp3}
\rho_{\varepsilon}(x)= x+\sum^n_{j=1} \varepsilon^j p_j(x). 
\end{equation}
We are interested in the distribution of $\rho_{\varepsilon}(Z)$ for small $\varepsilon$, where $Z\sim N(0,\Sigma)$. If we assume that $\rho_{\varepsilon}$ is bijective then this distribution has a density given by
\begin{equation*}\label{e-pp1}
f_{\varepsilon}(y) = \textup{det}\bra{D \rho^{-1}_{\varepsilon}(y)}\phi_\Sigma\bra{\rho_{\varepsilon}^{-1}(y)}. 
\end{equation*}
As Davie notes in the introduction of \cite{davie2015poly}, in general bijectivity will only hold on some bounded region of $\RR^d$ (which will be large if $\varepsilon$ is small). It actually turns out that bijectivity is a sufficient condition for (\ref{e-pp3}) to imply the following asymptotic expansion of the density:
\begin{equation}\label{e-pp2}
f_{\varepsilon}(y) = \phi_\Sigma(y)\bra{1+\sum^\infty_{j=1} \varepsilon^j S_j(y)},
\end{equation}
for some sequence of polynomials $\{S_j\}_{j=1}^\infty \subset P_\Sigma$. 

In fact, we can explicitly construct the sequence $\{S_j\}_{j=1}^\infty$ from the polynomial functions $\{p_j\}_{j=1}^n$ via a bijection. We introduce the notation $\mathcal{P}$ for the set of all sequences $(u_1,u_2,\ldots,)$ with $u_j \in \mathcal{P}$ with similar definitions for $\mathcal{P}^d, \mathcal{P}_\Sigma, \mathcal{P}_G^d$. 
By using the inverse linear mapping $\Ll_\Sigma^{-1} : P_\Sigma \to P^d_G$, Davie inductively constructs a bijection $\mathcal{S}_\Sigma : \mathcal{P}^d _G\to \mathcal{P}_\Sigma$ such that
\[
\mathcal{S}_\Sigma(p_1,\ldots,p_n, 0, 0,\ldots) = (S_1,S_2,\ldots).  
\]
Since each $S_k$ is dependent only on $p_1,\ldots,p_k$, this can be rewritten in the more succinct form of the truncated mapping: $\mathcal{S}^{(n)}_\Sigma(p_1,\ldots,p_n) = (S_1,\ldots,S_n)$. 
For the explicit definition of $\mathcal{S}_\Sigma$ we refer to \cite[Lemma 2]{davie2015poly}. 
We are now in a position to state the main result of \cite{davie2015poly} in a simplified form for the special case of $\Sigma=I_d$. 

\begin{Proposition}\cite[Proposition 1]{davie2015poly}\label{p-davie-poly-1}
Let $\{S_j\}_{j=1}^n \subset P_I$ and define the corresponding sequence $\{p_j\}_{j=1}^n \subset P^d_G$ via the truncated bijection 
\[
\mathcal{S}^{(n)}_I(p_1,\ldots,p_n)=(S_1,\ldots, S_n).
\]
Set $R$ to be an upper bound on the absolute values of the coefficients of $p_1,\ldots,p_n$. 
Using these sequences, define the mapping $\rho_{\varepsilon}$ as in (\ref{e-pp3}) and let $\nu_{\varepsilon,n}$ be the signed measure on $\RR^d$ with density (\ref{e-pp2}). Finally, let $\mu_\varepsilon$ be the law of $\rho_\varepsilon(Z)$ where $Z\sim N(0,I)$. Then for all $p\in [1,\infty)$ we have 
\[
\int_{\RR^d} \bra{1+\abs{y}}^p d\abs{\mu_{\varepsilon}-\nu_{\varepsilon,n}}(y) \leq C_p \varepsilon^{n+1}, 
\]
for some constant $C_p>0$ depending only on $d,n,R$ and the maximum degree of $p_1,\ldots,p_n$. 
\end{Proposition}

Armed with the latter result, we are now able to present the proof of Proposition \ref{p-poly-perturb-bound}. 

\begin{proof}[Proof of Proposition \ref{p-poly-perturb-bound}]
Following the same technique in \cite[Corollary 2]{davie2014kmt}, we first show that we may assume that $S_j \in P_\Sigma$ for all $j\in\{n_0,\ldots,n\}$. Indeed, let
\[
\beta_j:= \varepsilon^j\int_{\RR^d} S_j(y) \phi(y)\, dy \textup{ and }  \beta:=\sum_{j=n_0}^n \beta_j.
\]
Then we have 
\begin{align*}
\abs{\beta} 
= \abs{\sum_{j=n_0}^n \varepsilon^j \int_{\RR^d} S_j(y) \phi(y)\, dy}
&= \abs{ \int_{\RR^d} \phi(y)\bra{1+\sum_{j=n_0}^n \varepsilon^j S_j(y)}\, dy - \int_{\RR^d} \mu_\varepsilon(y)\, dy}\\
&= \abs{ \int_{\RR^d} \bra{\nu_{\varepsilon,n}-\mu_\varepsilon}(y)\, dy}\\
&\leq \int_{\RR^d} d\abs{\nu_{\varepsilon,n}-\mu_\varepsilon}(y)
\leq \int_{\RR^d} \bra{1+\abs{y}}^p  d\abs{\nu_{\varepsilon,n}-\mu_\varepsilon}(y) 
\leq \delta.
\end{align*}
Next, define a new signed measure $\tilde{\nu}_{\varepsilon,n}$ on $\RR^d$ with density
\[
\phi(y)\bra{1+\sum_{j=n_0}^n \bra{\varepsilon^j S_j(y)-\beta_j}} = \phi(y)\bra{1+\sum_{j=n_0}^n \varepsilon^j S_j(y)}  - \beta\phi(y). 
\]
Consequently, 
\begin{align*}
\int_{\RR^d} \bra{1+\abs{y}}^p d\abs{\tilde{\nu}_{\varepsilon,n}-\mu_\varepsilon}(y)
&\leq \int_{\RR^d} \bra{1+\abs{y}}^p \seq{d\abs{\nu_{\varepsilon,n} - \mu_\varepsilon}(y) + \abs{\beta} \phi(y)\, dy}\\
& \leq \delta + \abs{\beta} \int_{\RR^d} \bra{1+\abs{y}}^p \phi(y)\, dy
\leq C_p\delta.
\end{align*}
Therefore, by replacing each $S_j$ by $S_j-\beta_j$ we may assume that $S_j \in P_\Sigma$.

Since we can assume each $S_{j} \in P_I$, we can apply the bijection $\mathcal{S}_I$ of \cite[Lemma 2]{davie2015poly} to find a corresponding sequence $\{p_j\}_{j=1}^n \subset P^d_G$ such that 
\[
\mathcal{S}_I^{(n)}(p_1,\ldots,p_n) = (S_1,\ldots,S_n). 
\] 
As stated in the proof of \cite[Corollary 3]{davie2014kmt}, if $g_0=\Ll_I p_0$ for $p_0 \in P^d$, $g_0 \in P$, then $\textup{deg}(p_0) \leq \textup{deg}(g_0)$ and the absolute values of the coefficients of $p_0$ are bounded by those of $g_0$ up to a universal multiplicative constant. 
Thus by the recursive construction of the map $\mathcal{S}_I$ using $\Ll_I^{-1}$ in \cite{davie2015poly}, we can bound the absolute values of the coefficients of $p_1,\ldots,p_n$ by some function of those of $S_1,\ldots,S_n$. Similarly, the degree of each $p_j$ can be bounded by a function of $s=\max_{j} \textup{deg}(S_j)$. A simple consequence is the bound $\abs{p_j(Z)}_{L^p} \leq C_{d,M,p,s}$ for each $j$, where $Z\sim N(0,I)$. Taking the trivial coupling of $\Ll(Z)$ and $\mu_\varepsilon$, specifically $(Z,\rho_\varepsilon(Z))$, this leads to the estimate
\begin{align}
\Ww_p\bra{\mu_{\varepsilon},\phi} 
\leq \abs{Z-\rho_{\varepsilon}(Z)}_{L^p} 
\leq \sum_{j=n_0}^n \varepsilon^j \abs{p_j\bra{Z}}_{L^p}
&\leq C_{d,M,p,s} \sum_{j=n_0}^n \varepsilon^{j}\notag\\
&= C_{d,M,n,p,s}\varepsilon^{n_0}.\label{e-triangle-2}
\end{align}

As before, define the polynomial perturbation $\rho_{\varepsilon}$ by (\ref{e-pp3}) and let $\mu_{\varepsilon}$ be the law of $\rho_{\varepsilon}(Z)$ for $Z\sim N(0,I)$. Then Proposition \ref{p-davie-poly-1} guarantees that for all $p\in [1,\infty)$, 
\[
\int_{\RR^d} \bra{1+\abs{y}}^p d\abs{\mu_{\varepsilon}-\nu_{\varepsilon,n}}(y) \leq C_{d,M,n,p,s} \varepsilon^{n+1}. 
\]
Combined with (\ref{e-hyp-estimate}), this estimate gives
\[
\int_{\RR^d} \bra{1+\abs{y}}^p d\abs{\mu-\mu_{\varepsilon}}(y) \leq {\delta}+ C_{d,M,n,p,s} \varepsilon^{n+1},
\]
and so Lemma \ref{l-useless-lemma} yields the Wasserstein estimate
\begin{equation}\label{e-triangle-1}
\Ww_p\bra{\mu,\mu_{\varepsilon}} \leq C_{d,M,n,p,s}\bra{{\delta} + \varepsilon^{n+1}}^{1/p} 
\leq C_{d,M,n,p,s} \bra{\delta^{1/p} + \varepsilon^{\frac{n+1}{p}}}.
\end{equation} 
Finally, combining (\ref{e-triangle-2}) and (\ref{e-triangle-1}) via the triangle inequality, we conclude that
\begin{align*}
\Ww_p\bra{\mu,\phi}
&\leq \Ww_p\bra{\mu,\mu_{\varepsilon}} + \Ww_p\bra{\mu_{\varepsilon},\phi} 
\leq C_{d,M,n,p,s}\bra{ \varepsilon^{n_0} + {\delta}^{1/p} +  \varepsilon^{\frac{n+1}{p}}}.
\end{align*}
The proof is complete. 
\end{proof}

\begin{Remark}
The proof of Proposition \ref{p-poly-perturb-bound} reveals the reason why we insist on bounding the quantity $\int_{\RR^d} \bra{1+\abs{y}}^p d\abs{\mu-\nu}(y)$ rather than the simpler integral 
\[
\int_{\RR^d} \abs{y}^p d\abs{\mu-\nu}(y). 
\]
At first glance the latter quantity is all that is needed for Lemma \ref{l-useless-lemma} to bound $\Ww_p(\mu,\nu)$. However, in order to generalise the proof to polynomials not contained in $P_\Sigma$ we use the simple inequality 
\[
\int_{\RR^d} d\abs{\nu_{n,\varepsilon}-\mu}(y) \leq \int_{\RR^d}\bra{1+\abs{y}}^p d\abs{\nu_{n,\varepsilon}-\mu}(y). 
\]
This bound does not remain true when we replace the right-hand side with $\int_{\RR^d} \abs{y}^p d\abs{\nu_{n,\varepsilon}-\mu}(y)$. For instance, consider the case when the supports of $\nu_{n,\varepsilon}$ and $\mu$ are contained within the open unit ball around the origin of $\RR^d$.  
\end{Remark}

\section{Main coupling theorem}

We restate the main coupling theorem of the paper. Then, after establishing some notation and performing a linear transformation, we show that it suffices to prove the simpler statement of Proposition \ref{p-coupling-y-z}. 

\begin{Theorem}
For every $p\in [1,\infty)$, there exists a constant $C_p>0$ such that the following holds: given $N\in\NN$, we can construct a coupling between the i.i.d.~sequence $\{A^{(j)}\}_{j=0}^{N-1}$ and i.i.d.~$\{B^{(j)}\}_{j=0}^{N-1}$ random variables, with each pair $(A^{(j)}, B^{(j)})$ defined by (\ref{l-area-decomposition}) and (\ref{e-b-def}) using a common Brownian increment $W^{(j)} \sim N(0,N^{-1}I_d)$, such that 
\[
\max_{r=1,\ldots,N} \abs{\norm{ \sum^{r-1}_{j=0} \bra{A^{(j)}-B^{(j)}}}_{\RR^{\frac{d}{2}(d-1)}}}_{L^p} \leq C_p \frac{\log N}{N}. 
\]
\end{Theorem}

Before presenting the proof in the Section \ref{s-coupling-proof}, we first set up some notation, then consider a version of the central limit theorem in the next section. We fix a sequence of independent Brownian increments $\{W^{(j)}\}_{j=0}^{N-1}$ with $W^{(j)} \sim N(0,N^{-1}I_d)$, and let $\Gg$ denote the $\sigma$-algebra generated by these variables. 

For each $r\in \{0,1,\ldots,N-1\}$ define the random vector $X^{(r)} \in \RR^{d_1}$ by 
\begin{align*}
X^{(r)}_k &:= \sqrt{12N}\zeta^{(r)}_k \textup{ for } 1\leq k \leq d;\\
 X^{(r)}_{\frac{k}{2}(2d-k-1)+l} &:= \sqrt{12N^2} K^{(r)}_{kl} \textup{ for } 1\leq k < l \leq d. 
\end{align*}
Then, conditional on $\Gg$, $X^{(r)}$ has mean zero and covariance matrix $I_{d_2}$. We can then write $A^{(r)} = N^{-1} G_r X^{(r)}$, where $G_r$ is a $d_2 \times d_1$ matrix defined in terms of the $W^{(j)}$. Specifically,
\[
G_{r}=\frac{1}{\sqrt{12}}\left(\begin{array}{c|c}
M_{r} & I_{d_2}\end{array}\right), 
\]
where $M_r$ is the $d_2 \times d$ matrix defined by the rows
\[
\bra{M_r}_{\frac{k}{2}(2d-k-1)+(l-d)} = \sqrt{N}\bra{ W^{(r)}_l e_k - W^{(r)}_k e_l}.
\] 
This makes $M_r$ have the form:
\[
M_r = \sqrt{N}
\left(\begin{array}{cccccc}
W_{2}^{(r)} & -W_{1}^{(r)} & 0 & \cdots & 0 & 0\\
W_{3}^{(r)} & 0 & -W_{1}^{(r)} & \cdots & 0 & 0\\
\vdots & \vdots & \vdots & \ddots & \vdots & \vdots\\
0 & 0 & 0 & \cdots & W_{d}^{(r)} & -W_{d-1}^{(r)}
\end{array}\right).
\]
In the same way we can write $B^{(r)} = N^{-1}G_r \tilde{X}^{(r)}$, where $\tilde{X}^{(r)}$ has the normal distribution $N(0,I_{d_1})$. 

By scaling we can see that to complete the proof it is sufficient to construct a coupling of the random walks composed of the vectors $X^{(r)}$ and $\tilde{X}^{(r)}$, conditional on $\Gg$, such that 
\begin{equation}\label{e-g-x}
\max_{r=1,\ldots,N} \abs{ \norm{ \sum_{j=0}^{r-1} G_j\bra{X^{(j)} - \tilde{X}^{(j)} }}_{\RR^{\frac{d}{2}(d-1)}}}_{L^p} \leq C_p\log N
\end{equation}
for some constant $C_p>0$. 
To this end, we first note that without loss of generality we may assume $N=h^{-1}=2^m$ for some $m\in\NN$. We define a dyadic set to be a subset $E\subseteq \{0,1,\ldots,2^m-1\}$ of the form 
\[
E = \{k2^n, k2^n+1, \ldots, (k+1)2^n - 1\}, 
\]
for some integers $n\in\{0,1,\ldots,m\}$ and $k\in \{0,1,\ldots,2^{m-n}-1\}$. Since a consecutive set $F\subset \{0,1,\ldots,2^m-1\}$ can be expressed as the disjoint union of at most $\log_2 N$ dyadic subsets $E_1,\ldots,E_k$ of different sizes, we need only prove
\begin{equation}\label{e-g-x-2}
\EE\bra{ \norm{\sum_{r\in E} G_r \bra{X^{(r)} - \tilde{X}^{(r)}}}_{\RR^{\frac{d}{2}(d-1)}}^p} \leq C_p
\end{equation}
for all dyadic sets $E$ in order to establish (\ref{e-g-x}). 

Next, for each dyadic set $E$ of size $2^n$ let us define the matrix 
\begin{equation}\label{e-H-e-pos-def}
H_E := 2^{-n} \sum_{r\in E} G_r G_r^t
\end{equation}
along with the random variables
\[
Y_E:= 2^{-n/2} \sum_{r\in E} G_r X^{(r)} \textup{ and } Z_E:= 2^{-n/2}\sum_{r\in E} G_r \tilde{X}^{(r)}. 
\]
Since, conditional on $\Gg$, the random variables $\{A^{(j)}\}_{j=0}^{N-1}$ are independent, $H_E$ is the (conditional) covariance matrix of $Y_E$. Similarly, $H_E$ is also the (conditional) covariance matrix of $Z_E$. Note that $H_E^{-1}$ is well-defined since each product $G_rG_r^t$ is a positive-definite symmetric matrix. Indeed, block matrix multiplication confirms that 
\[
G_r G_r^t = \frac{1}{12}\bra{I_{d_2} + M_r M_r^t}.
\]
It also follows that the eigenvalues of $G_rG_r^t$ are bounded below by $\frac{1}{12}$ and so $\norm{(G_rG_r^t)^{-1}} \leq 12$. Similarly, $\norm{H_E^{-1}} \leq 12$ for all dyadic $E$. Moreover, since $\sqrt{N}W^{(r)}_k \sim N(0,1)$, certainly $\EE\bra{\norm{G_r}^p} \leq C_p$ for all $p\geq 1$. In fact, $\norm{G_r}^2$ possesses exponential tails (a property which we will exploit later in the coupling construction of Section \ref{s-coupling-proof}). 

\begin{Lemma}\label{l-h-exp-tails}
For all $\alpha \in (0,\frac{1}{48d})$ there exists a constant $C_{\alpha} >0$ such that
\[
\max_{r=0,1,\ldots,N-1}\EE\bra{e^{\alpha\norm{G_r}^2}} \leq C_{\alpha}
\]
Similarly, $\EE\bra{e^{\alpha\norm{H_E}}} \leq C_{\alpha}$ for all consecutive sets $E$. 
\end{Lemma}
\begin{proof}
Recall that the matrix norm $\norm{M_r}$ is bounded by the $l_2$-norm of its entries $\{M_r(i,j)\}_{i,j}$. Hence, 
\[
\norm{M_r}^2 \leq \sum_{i,j} \abs{M_r(i,j)}^2 = 2h^{-1}\sum^d_{k=1} \abs{W^{(r)}_k}^2,
\]
and so for all $\alpha \in (0,\frac{1}{4d})$ we have 
\[
\EE\bra{e^{\alpha \norm{M_r}^2}} \leq \prod_{k=1}^d \EE\bra{e^{2\alpha h^{-1}\abs{W_k^{(r)}}^2}} 
= \EE\bra{e^{2\alpha dh^{-1} \abs{W^{(r)}_k}^2}} = C_{\alpha} < \infty. 
\]
Thus for small enough $\alpha \in (0,\frac{1}{48d})$, 
\[
\EE\bra{e^{\alpha \norm{G_r}^2}} 
\leq \EE\bra{e^{\frac{\alpha}{12}\bra{\norm{M_r}+1}^2}}
\leq e^{\frac{\alpha}{6}}\EE\bra{e^{\frac{\alpha}{6}\norm{M_r}^2}} \leq C_{\alpha}. 
\]
The second statement then follows:
\begin{align*}
\EE\bra{e^{\alpha \norm{H_E}}} 
\leq 
\EE\,{\exp\bra{\alpha 2^{-n}\sum_{r\in E} \norm{G_r}^2}}
&=\prod_{r\in E} \EE\bra{e^{\alpha 2^{-n} \norm{G_r}^2}}
= \EE\bra{e^{\alpha \norm{G_r}^2}} = C_{\alpha}.\label{e-h-exp-tails}
\end{align*}
The proof is complete. 
\end{proof}

The aim is to prove the following proposition, from which (\ref{e-g-x-2}) follows immediately, thereby establishing Theorem \ref{t-coupling}.

\begin{Proposition}\label{p-coupling-y-z}
With the notation above, for every $p \in [1,\infty)$ there exists a constant $C_p>0$ and a coupling of $\{Y_E\}_{E \subseteq E_0}$ and $\{Z_E\}_{E \subseteq E_0}$, conditional on $\Gg$, such that 
\[
\abs{ \norm{Y_E - Z_E}_{\RR^{\frac{d}{2}(d-1)}}}_{L^p} \leq C_p 2^{-n/2}
\]
for every dyadic set $E$ of size $2^n$, where $n\leq m = \log_2N$.
\end{Proposition}

We copy the coupling construction of \cite[Theorem 1]{davie2014kmt} except for modifications which are needed to establish general $L^p$-estimates. This is because the original result only established the coupling for $p\in [1,4)$.
Thus we require the higher order polynomial perturbation result of Proposition \ref{p-poly-perturb-bound} along with Lemma \ref{l-davie-lemma-4-new} of the next section. Hence we postpone the proof of Proposition \ref{p-coupling-y-z} to Section \ref{s-coupling-proof}. 

\begin{Remark}
It is a common practice in Gaussian rough path theory to use Wiener-It\^{o} chaos to establish general $L^p$-bounds from the $p=2$ case (\cite[Theorem D.8]{FV}). 
Certainly each $B^{(j)}$ increment lives in the second inhomogeneous Wiener chaos (being a quadratic polynomial of Gaussian random variables), and it is well-known that L\'{e}vy area also lives in a (possibly different) second Wiener chaos \cite[Proposition 15.19]{FV}.
One could ask whether we could use this theory to immediately get all $L^p$-estimates from the $p=2\in [1,4)$ case. Alas  our coupling argument does not necessarily guarantee that the random variables $\{A^{(j)}\}_{j=0}^{N-1}$ and $\{B^{(j)}\}_{j=0}^{N-1}$ belong to the same Wiener chaos and so we cannot apply the theory to their difference. 
\end{Remark} 

\section{A central limit theorem}

In light of the previous polynomial perturbation extension of Proposition \ref{p-poly-perturb-bound}, we need a modified version of \cite[Lemma 4]{davie2014kmt} for the proof of Proposition \ref{p-coupling-y-z} in the next section. This is contained in the following lemma. It can be viewed as a version of the central limit theorem, stating that the density of $Y_E$ is close to the (Gaussian) density of $Z_E$ as the size of the dyadic set $E$ increases.  

\begin{Lemma}\label{l-davie-lemma-4-new}
Let $E$ be a dyadic set of size $2^n$ and let $f_E$ be the density function of $Y_E$, conditional on $\Gg$. Fix $\eta\in (0,\frac{1}{20})$ and an integer $\kappa\geq 2$. 
 Then, provided that $\norm{G_r} \leq 2^{n\eta}$ and $\norm{(G_rG_r^t)^{-1}} \leq 2^{2n\eta}$, there exists a constant $C_\kappa>0$ such that the following holds: for each $r\in E$ we have
\[
\abs{f_E - \phi_{H_E} \bra{1+ \sum_{k=2}^\kappa 2^{nk\bra{2\eta-\frac{1}{2}}} S_{E,2k}}  }(v) 
\leq C_\kappa2^{n\kappa\bra{10\eta -\frac{1}{2}}} \phi_{H_E}(v)
\]
for all $\abs{v}\leq 2^{n\eta}$. Here, each $S_{E,2k}\in P(\RR^{d_2})$ is respectively a polynomial of degree $2k$ with coefficients whose absolute values are bounded by some universal constant $C>0$ independent of $E$ and $k$.
\end{Lemma}

Note the original \cite[Lemma 4]{davie2014kmt} can be recovered by setting $\kappa=2$. 

\begin{proof}
The bounds on $G_r$ imply that $\norm{H_E} \leq 2^{2n\eta}$ and $ \norm{H_E^{-1}} \leq 2^{2n\eta}$. 

Let $\psi$ be the characteristic function of $X^{(r)}$ (which is independent of $r$). An explicit expression for $\psi$ can be found in \cite{wiktorsson2001joint}. Note that $\psi$ is real-valued and even on $\RR^{d_1}$ and extends to a complex-analytic function of a strip $\{x+iy : x,y\in \RR^{d_1},\, \abs{y}< a\}$ for some $a>0$. In a neighbourhood of $0$ in $\mathbb{C}^{d_1}$, $\log \psi$ has a convergent expansion 
\[
\log \psi(z) = -\frac{1}{2}\abs{z}^2 + c_4(z) + c_6(z) + \ldots,
\]
where $c_k(z)$ is a homogeneous polynomial of degree $k$ satisfying $\abs{c_k(z)} \leq (C\abs{z})^k$ for even $k\geq 4$. Thus, 
\[
\psi(z) = \exp\bra{-\frac{1}{2}\abs{z}^2 + \chi(z)} \textup{ where } \chi(z):= \sum_{k=2}^\infty c_{2k}(z).
\]
From this it follows that there exists some $\delta>0$ such that 
\begin{equation}\label{e-davie-20}
\textup{if } 
x,y\in \RR \textup{ with } 2\abs{y} \leq \abs{x} < \delta \textup{ then } \abs{\psi(x+iy)} \leq e^{-\frac{1}{6}\abs{x}^2}.
\end{equation}
Using the decay of $\psi(z)$ as $x=\textup{Re}(z) \to \infty$ and the fact that $\abs{\psi(x)} < 1$ for $0\neq x\in \RR^{d_1}$, we can find $\gamma \in (0,1)$ and $\delta^\prime >0$ so that 
\begin{align}
\textup{if } x,y\in\RR^{d_1} \textup{ with } \abs{x}\geq \delta &\textup{ and } 
\abs{y} \leq \delta^\prime
\textup{ then } \abs{\psi(x+iy)} \leq \textup{min}\bra{\gamma, C\abs{x}^{-1}}.\label{e-davie-21}
\end{align}
Now let $\Psi$ be the characteristic function of $Y_E$; $\Psi(u) = \prod_{r\in E} \psi\bra{2^{-n/2}G_r^t u}$. 
Recall that $\Psi$ is the Fourier transform of the density $f_E$ of $Y_E$, and taking the inverse Fourier transform we obtain the expression 
\[
f_E(v) = \bra{2\pi}^{-d_2/2} \int_{\RR^{d_2}} e^{-iu^tv}\Psi(u)\, du.
\] 
By translating the subspace of integration in $\mathbb{C}^{d_2}$ by $-iH_E^{-1}v$ we can rewrite this as 
\begin{align*}
f_E(v) &= \bra{2\pi}^{-d_2/2} e^{-v^t H_E^{-1}v} \int_{\RR^{d_2}} e^{-iu^tv} \Psi\bra{u-iH_E^{-1}v}\, du\\
&= \bra{\textup{det}\, H_E}^{1/2} \phi_{H_E}(v) \int_{\RR^{d_2}} e^{-iu^tv} \Psi\bra{u-iH_E^{-1}v}\, du. 
\end{align*}
Note that 
\begin{equation}\label{e-davie-prod}
\Psi\bra{u-iH_E^{-1}v} = \prod_{r\in E} \psi\bra{2^{-n/2}G_r^t u - i 2^{-n/2}G_r^t H_E^{-1}v}.
\end{equation}
If $\abs{u} \geq 2^{4n\eta + 1}$ then using (\ref{e-davie-20}) and (\ref{e-davie-21}) we see that each term in the product is bounded by either $\min\bra{\gamma,  C2^{n(\eta+1/2)}\abs{u}^{-1}}$ or $\exp\bra{-\frac{1}{6}2^{-n(1+2\eta)}\abs{u}^2}$. Consequently, the product (\ref{e-davie-prod}) is bounded by 
\[
\abs{\Psi\bra{u-iH_E^{-1}v}} \leq \min\bra{\gamma, C2^{n(\eta+1/2)}\abs{u}^{-1}}^{2^n} + \exp\bra{-\frac{1}{6}2^{-2n\eta }\abs{u}^2}
\]
for all $\abs{u} \geq 2^{4n\eta + 1}$. It then follows that 
\[
\int_{\seq{\abs{u}\geq 2^{4 n\eta + 1}}} \abs{\Psi\bra{u-iH_E^{-1}v}}\, du 
\leq C\seq{ 2^{nm}\gamma^{2^n} + \exp\bra{-2^{6 n\eta - 1} }}.
\]

To consider the case of $\abs{u}\leq 2^{4 n\eta + 1}$, we first set $w=u-iH_E^{-1}v$. We then have
\begin{align*}
e^{-iu^tv} \Psi(w) 
&=
e^{-iu^tv} \prod_{r\in E} \psi\bra{2^{-n/2} G_r^t w}\\
&= 
e^{-iu^tv} \exp\bra{\sum_{r\in E} \seq{-\frac{1}{2} 2^{-n} \abs{G_r^t w}^2 + \chi\bra{2^{-n/2}G_r^t w}}}\\
&= e^{-iu^t v}\exp\bra{\frac{1}{2}v^tH_E^{-1}v - \frac{1}{2}u^tH_E u + iu^tv + \Lambda(w)}\\
&=\exp\bra{\frac{1}{2}v^tH_E^{-1}v - \frac{1}{2}u^tH_E u + \Lambda(w)},
\end{align*}
where 
\[
\Lambda(w)=\sum_{r\in E} \chi\bra{2^{-n/2}G_r^tw} = \sum_{k=2}^\infty T_{2k}(w) \textup{ with } T_{2k}(w)=2^{-kn} \sum_{r\in E} c_{2k}\bra{G_r^t w}. 
\]
We see that $T_{2k}$ is a homogeneous polynomial of degree $2k$ and satisfies 
\[
\abs{T_{2k}(w)} \leq C2^{n(1 + k(2\eta-1))}\abs{w}^{2k}.
\] 

Next, we approximate $e^{\Lambda(w)}$ by an inhomogeneous polynomial of degree $2\kappa$ composed of $\{T_{2k}\}_{k=2}^\kappa$ and their powers (as mentioned previously, the original proof of \cite{davie2014kmt} set $\kappa=2$). First note that 
\begin{align*}
\abs{\sum_{k=\kappa+1}^\infty T_{2k}(w)} \leq \sum^\infty_{k=\kappa+1} \abs{T_{2k}(w)} 
&\leq C \sum^\infty_{k=\kappa+1} 2^{n(1 + k(2\eta-1))}\abs{w}^{2k}\\
&= C2^n\sum^\infty_{k=\kappa+1} \bra{2^{(2\eta-1)n}\abs{w}^2}^k\\
&\leq C2^n \bra{2^{(2\eta-1)n}\abs{w}^2}^{2(\kappa+1)}\\
&= C2^{n\bra{1+2(\kappa+1)(2\eta-1)}}\abs{w}^{4(\kappa+1)}
\end{align*}
and so the simple Taylor approximation $e^x = 1+ O(x)$ for small $x\geq 0$ gives 
\begin{align}
e^{\Lambda(w)} 
= \prod_{k=2}^\infty e^{T_{2k}(w)} 
&=  \prod_{k=2}^\kappa e^{T_{2k}(w)} \cdot \exp\bra{\sum^\infty_{k=\kappa+1} T_{2k}(w)}\notag\\
&= \prod_{k=2}^\kappa e^{T_{2k}(w)} \seq{ 1 + O\bra{2^{n\bra{1+2(\kappa+1)(2\eta-1)}}\abs{w}^{4(\kappa+1)}}}.\label{e-product-big}
\end{align}
We now consider each $e^{T_{2k}(w)}$ term in the product. For each $k$ and integer $l\geq 1$, 
\begin{align*}
e^{T_{2k}(w)} -1 =  \sum^\infty_{i=1} \frac{1}{i!}T_{2k}(w)^i 
&= \sum_{i=1}^{l} 
\frac{1}{i!}T_{2k}(w)^i + O\bra{\abs{T_{2k}(w)}^{l+1}}\\
&= \sum_{i=1}^{l} \frac{1}{i!}T_{2k}(w)^i + O\bra{2^{n\bra{1+k(2\eta-1)}(l+1)} \abs{w}^{2k(l+1)}}. 
\end{align*}
Since $k\geq 2$, $1-k\leq -\frac{k}{2}$ then $(1+k(2\eta-1)) \leq k\bra{2\eta-\frac{1}{2}}$. Hence, if for each $k$ we choose $l_k$ to be the smallest integer such that $k(l_k+1)\geq \kappa$, we have 
\begin{align}
e^{T_{2k}(w)} &= 1 + \sum^{l_k}_{i=1} \frac{1}{i!} T_{2k}(w)^i 
+ O\bra{2^{nk\bra{2\eta - \frac{1}{2}}(l_k+1)}\abs{w}^{2k(l_k+1)}}\notag\\
&= 1 + \sum^{l_k}_{i=1} \frac{1}{i!} T_{2k}(w)^i 
+ O\bra{2^{n\kappa\bra{2\eta-\frac{1}{2}}}\abs{w}^{2\kappa}}.\label{e-little-exp}
\end{align}
Combining (\ref{e-product-big}) and (\ref{e-little-exp}), it follows that 
\[
e^{\Lambda(w)} = 1 + \sum^{\kappa}_{k=2} q_{2k}(w) + O\bra{2^{n\kappa \bra{2\eta-\frac{1}{2}}}\abs{w}^{2\kappa}}
\]
for some homogeneous polynomials $\{q_{2k}\}_{k=2}^{\kappa}$, where each $q_{2k}$ has degree $2k$ and coefficients bounded by $C2^{nk \bra{2\eta-\frac{1}{2}}}$. 
Consequently, setting $Q_{2\kappa}(w):= \sum^\kappa_{k=2} q_{2k}(w)$, we have 
\begin{align*}
&e^{-\frac{1}{2}v^tH_E^{-1}v} \int_{\seq{\abs{u}\leq 2^{4 n\eta +1}}} 
\abs{ e^{-iu^tv}\Psi\bra{u-iH_E^{-1}v} - \bra{1+Q_{2\kappa}\bra{u-iH_E^{-1}v}}e^{-\frac{1}{2}u^t H_E u}}\, du\\
&\gap\gap
= e^{-\frac{1}{2}v^tH_E^{-1}v} \int_{\seq{\abs{u}\leq 2^{4 n\eta +1}}} 
\abs{e^{\Lambda(u-iH_E^{-1}v)} - 1 - Q_{2\kappa}(u-iH_E^{-1}v)}e^{-\frac{1}{2}u^tH_E u}\, du\\
&\gap\gap
\leq C2^{n\kappa\bra{2\eta-\frac{1}{2}}}  e^{-\frac{1}{2}v^tH_E^{-1}v}
\int_{\seq{\abs{u}\leq 2^{4 n\eta +1}}} 
e^{-\frac{1}{2}u^tH_Eu} \abs{u-iH_E^{-1}v}^{2\kappa}\, du\\
&\gap\gap=
C2^{n\kappa\bra{2\eta-\frac{1}{2}} + 8n\kappa}
= C2^{n\kappa\bra{10\eta-\frac{1}{2}}}
\end{align*}
where we have used the inequality $\abs{iH_E^{-1}v} \leq \norm{H_E^{-1}}\abs{v} \leq 12\abs{v}$. 
Moreover, 
\begin{align*}
\int_{\seq{\abs{u}\geq 2^{4 n\eta + 1}}} \abs{1+Q_{2\kappa}(u-iH_E^{-1}v)}
e^{-\frac{1}{2}u^tH_Eu}\, du
&\leq Ce^{-2^{n\eta}}.
\end{align*}
Collecting these bounds, the lemma then follows by setting 
\begin{align*}
S_{E,2k} := 2^{-nk\bra{2\eta-\frac{1}{2}}}R_{E,2k}
\textup{ where } 
R_{E,2k}(v):= \int_{\RR^{d_2}} q_{2k}(u-iH_E^{-1}v)e^{-\frac{1}{2}u^tH_E u}\, du,
\end{align*} 
for each $k\in\{2,\ldots,\kappa\}$. 
The polynomial $R_{E,2k}$ is of degree $2k$ with coefficients bounded by $C2^{nk\bra{2\eta-\frac{1}{2}}}$. 
The proof is complete. 
\end{proof}

\section{Proof of the coupling construction}\label{s-coupling-proof}

We are now in a position to prove Proposition \ref{p-coupling-y-z}, which in turn establishes Theorem \ref{t-coupling}.
Throughout the proof we replace the bulky notation of $\abs{\norm{Y_E-Z_E}}_{L^p}$ with $\abs{Y_E-Z_E}_{L^p}$.

\begin{proof}[Proof of Proposition \ref{p-coupling-y-z}]
The idea of Davie's original proof in \cite{davie2014kmt} is to construct a coupling of $Y_E$ and $Z_E$ recursively, starting with the base case $E_0 = \{0,1,\ldots,2^m-1\}$ and proceeding by successive bisection to smaller dyadic sets.

To begin the construction, we fix a constant $\eta\in (0,\frac{1}{44})$ and let $\kappa$ be the smallest integer such that $\kappa\geq 2p$. It follows that any constant dependent upon both $\kappa$ and $p$ can be made solely dependent on $p$. Intuitively, the inequality $\kappa\geq 2p\geq 2$ makes sense. If we desire a  stronger $L^p$-bound, the degree ($2\kappa$) of the polynomial approximation produced by Lemma \ref{l-davie-lemma-4-new} will have to increase. 

\textbf{Initial step.} 
We start the construction by finding a coupling between $Y_{E_0}$ and $Z_{E_0}$. 
To this end, let $\Ee_0$ be the event $\{ \norm{G_r} \leq 2^{m\eta} : \textup{ for all } r\in E_0\}$. Note that under $\Ee_0$, $\norm{H_{E_0}} \leq 2^{2m\eta}$ (and recall that $\norm{H_{E_0}^{-1}} \leq 12$ regardless). Thus provided $\Ee_0$ holds, we apply Lemma \ref{l-davie-lemma-4-new} to find that for all $\abs{y} \leq 2^{m\eta}$, 
\begin{equation}\label{e-davie-est}
\abs{\frac{f_{E_0}(y)}{\phi_{H_{E_0}}(y)} - \seq{1+r_{E_0}}(y)} \leq C_\kappa2^{m\kappa\bra{10\eta-\frac{1}{2}}}, 
\end{equation}
where
\[
r_{E_0}:=\sum_{k=2}^\kappa 2^{nk\bra{2\eta-\frac{1}{2}}}S_{E_0,2k}. 
\]
Here, each $S_{E_0,2k} \in P(\RR^{d_2})$ is a polynomial of degree $2k$ for which the absolute values of its coefficients are bounded by some constant independent of $\eta,H_{E_0},k$. 
We write $y=H_{E_0}^{1/2}u$ and define the probability density $h(u)=(\textup{det} \,H_{E_0})^{1/2} f_{E_0}(H^{1/2}_{E_0} u)$. To convince ourselves that $h$ is a probability density, we observe that the Jacobian matrix of the linear transformation $y\mapsto H_{E_0}^{-1/2}y$ is the matrix $H_{E_0}^{-1/2}$ itself and so the chain rule gives
\[
f_{E_0}(y)= \bra{\textup{det}\, H_{E_0}^{-1/2}} h\bra{H^{-1/2}_{E_0}y} =
(\textup{det}\, H_{E_0}^{1/2})^{-1} h(u) 
= \bra{\textup{det}\, H_{E_0}}^{-1/2} h(u). 
\]
The matrix transformation also gives
\begin{align*}
\phi_{H_{E_0}}\bra{y}
&= \bra{\textup{det} \, H_{E_0}}^{-1/2} \frac{1}{\bra{2\pi}^{d_2/2}}\exp\bra{-\frac{1}{2}y^t H_{E_0}^{-1} y}
= \bra{\textup{det} \, H_{E_0}}^{-1/2} \phi(u)
\end{align*}
and this change of variables immediately implies the inequality 
\begin{equation}\label{e-great-bound}
\Ww_p\bra{f_{E_0}, \phi_{H_{E_0}}} \leq \norm{H_{E_0}^{1/2}}\Ww_p\bra{h, \phi}.
\end{equation}
Setting $\mathcal{A}=\{u\in \RR^{d_2} : \abs{H^{1/2}_{E_0}u}\leq 2^{m\eta}\}$, it follows from (\ref{e-davie-est}) that  
\begin{align}
\int_{\mathcal{A}}  \abs{h(u)-\seq{1+r_{E_0}\bra{H^{1/2}_{E_0}u}}\phi(u)}\, du 
&\leq C_\kappa2^{m\kappa\bra{10\eta-\frac{1}{2}}} \int_{\Aa} \bra{1+\abs{u}^p} \phi(u)\, du\notag\\
&\leq C_p 2^{m\kappa\bra{10\eta-\frac{1}{2}}}.\label{e-a-est} 
\end{align}
The exponential tail property of the Gaussian distribution ensures
\begin{align*}
\int_{\Aa^c}  \bra{1+\abs{u}}^p  \seq{1+r_{E_0}\bra{H^{1/2}_{E_0}u}}\phi(u)\, du 
&\leq C_p \bra{1 + \norm{H_{E_0}}^{\kappa}} \int_{\Aa^c} \abs{u}^{2\kappa} \phi(u)\, du\\
&\leq C_p 2^{2m\kappa\eta} e^{-\alpha 2^{m\eta}}
\leq C_p 2^{-\beta2^{m\eta}}
\end{align*}
for some constants $\alpha>\beta>0$. Similarly, the density $h$ conditional on $\Gg$ possesses exponential tails and so
$\int_{\Aa^c}  \bra{1+\abs{u}}^p h(u)\, du \leq C_p2^{-\beta 2^{m\eta}}$. Thus the integral of (\ref{e-a-est}) can be taken over all of $\RR^{d_2}$ with the inequality remaining true. 

In light of (\ref{e-a-est}), (now taken over $\RR^{d_2}$), we now apply Proposition \ref{p-poly-perturb-bound} with $\Sigma=I_{d_2}$, $\delta=C2^{m\kappa\bra{10\eta-\frac{1}{2}}}$, $K\leq C$, $n_0=2$, $n=\kappa$ and $\varepsilon=2^{m\bra{2\eta-\frac{1}{2}}}$ to find that conditional of $\Ee_0$, 
\begin{align*}
\Ww_p\bra{h,\phi} &\leq C_{\kappa,p} \bra{  2^{2m\bra{2\eta-\frac{1}{2}}} + 2^{m\frac{\kappa}{p}\bra{10\eta-\frac{1}{2}}} + 2^{m\frac{\kappa+1}{p}\bra{2\eta-\frac{1}{2}}}}
\leq C_{\kappa,p} 2^{m\bra{21\eta-1}}. 
\end{align*}
Here, we have used the fact that $\frac{\kappa}{p}\geq 2$. Recalling (\ref{e-great-bound}), we deduce that
\begin{equation}\label{e-new-was-bound-2}
\Ww_p\bra{f_{E_0}, \phi_{H_{E_0}}} \leq \norm{H_{E_0}^{1/2}} \Ww_p\bra{h,\phi} \leq C_{\kappa,p} 2^{m\bra{22\eta-1}}
\end{equation}
since under $\Ee_0$ we have $\norm{H_{E_0}} \leq 2^{2m\eta}$. 
To finish the initial step, we recall the exponential tail bounds of $\norm{G_r}^2$ given by Lemma \ref{l-h-exp-tails} to deduce that for all $\alpha \in (0,\frac{1}{48d})$:
\begin{align*}
\PP\bra{\Ee_0^c} 
\leq \sum_{r\in E_0} \PP\bra{\norm{G_r} \geq 2^{ m\eta}}
&\leq 2^m\PP\bra{e^{\alpha\norm{G_0}^2} \geq e^{\alpha 2^{2 m\eta}}}\\
&\leq 2^m e^{-\alpha 2^{2m\eta}} \EE\bra{e^{\alpha \norm{G_0}^2}}\\
&=  C_{\alpha}\exp\bra{-\alpha 2^{2 m\eta} + m\log 2}.
\end{align*}
If $\Ee_0$ fails we simply construct independent copies of $Y_{E_0}$ and $Z_{E_0}$. Hence, from (\ref{e-new-was-bound-2}) and H\"{o}lder's inequality, we can find a coupling of $Y_{E_0}$ and $Z_{E_0}$ such that unconditionally
\begin{align*}
\EE \bra{\abs{Y_{E_0}-Z_{E_0}}^p}
&\leq \EE \bra{\abs{Y_{E_0}-Z_{E_0}}^p\id_{\Ee_0}} + \EE\bra{ {\abs{Y_{E_0}  + Z_{E_0}}^p}\id_{\Ee_0^c}}
\leq C_{\kappa,p}2^{mp\bra{22\eta-1}}
\end{align*}
for some constant $C_{\kappa,p}>0$. 

\textbf{Recursive step.} 
Let $E$ be a dyadic set of size $2^n$. Then we can write $E=F\cup G$ where $F$ and $G$ are disjoint dyadic sets of size $2^{n-1}$. Note that
\[
Y_F + Y_G = 2^{1/2}Y_E \textup{ and } Z_F+Z_G=2^{1/2}Z_E. 
\]
We suppose a coupling between $Y_E$ and $Z_E$ has been defined, conditional on $\Gg$. In other words, for each choice of $\{W^{(j)}\}_{j=0}^{N-1}$, we have a joint distribution of $Y_E$ and $Z_E$ with the correct conditional marginal distributions. We wish to extend this coupling to a coupling between $(Y_F,Y_G)$ and $(Z_F,Z_G)$. 

For each $x\in \RR^{d_2}$, let $f_x$ denote the the density of $Y_E$ conditional on $Y_E=x$ and on $\Gg$. Similarly, let $g_x$ be the corresponding density for $Z_E$. As noted in \cite{davie2014kmt}, the conditional distribution of $Z_F$, given $Z_E=x$ and $\Gg$, is $N(Jx,H)$ where $J=H_FH_E^{-1}$ and $H=\frac{1}{2}H_FH_E^{-1}H_G$. Thus $g_x$ is the density of $N(Jx,H)$. 

We need to find a coupling between $Y_F$ and $Z_F$ conditional on $Y_E=x$ and $Z_E=\tilde{x}$. To do this we need a coupling between the distributions with densities $f_x$ and $g_{\tilde{x}}$. However, we shall instead construct a coupling between $f_x$ and $g_x$, then use the fact that $g_{\tilde{x}}$ is just $g_x$ translated by $J(x-\tilde{x})$. 


We begin by noting that $f_x(y)=\frac{2^{1/2}f_F(y)f_G(2^{1/2}x-y)}{f_E(x)}$. Then provided the event $\Ee = \{\norm{G_r} \leq 2^{n\frac{\eta}{2p}} : \textup{for all } r\in E\}$ holds, we apply Lemma \ref{l-davie-lemma-4-new} to each of $E,F,G$ to find:
\begin{equation}\label{e-davie-28}
\abs{\frac{f_x(y)}{g_x(y)} - \seq{1+r_x(y)}g_x(y)} \leq C_\kappa2^{n\kappa\bra{10\eta-\frac{1}{2}}} \textup{ for all } \abs{x},\abs{y} \leq 2^{n{\eta}/{p}},
\end{equation}
where $r_x(y)= r_F(y)+r_G(2^{1/2}x-y)-r_E(x)$. Note that 
\[
r_E(y)=\sum^\kappa_{k=2} 2^{nk\bra{2\eta-\frac{1}{2}}} S_{E,2k},
\] 
where each $S_{E,2k}\in P(\RR^{d_2})$ is a polynomial of degree $2k$ with the absolute value of its coefficients bounded by some constant independent of $\eta,H,k$.
Corresponding decompositions and associated properties hold for $r_F$ and $r_G$. 

Next, we define the set 
\[
\Omega := \seq{x\in \RR^{d_2} : \EE\bra{\abs{Y_F}^p \id_{\seq{\abs{Y_F}\geq 2^{n\eta/p }}} \big| \,Y_E=x \textup{ and } \Gg} > 2^{-n\kappa/2 }}. 
\]
Utilising the inequalities of Markov and H\"{o}lder gives 
\begin{align*}
\PP\bra{Y_E\in\Omega \big| \Gg}
&\leq 2^{n\kappa/2}\EE_\Gg\bra{\abs{Y_F}^p \id_{\seq{\abs{Y_F}\geq 2^{n\eta /p}}}}\\
&\leq2^{n\kappa/2 } \EE_\Gg\bra{\abs{Y_F}^{pq}}^{1/q} \PP\bra{\abs{Y_F}^{s} \geq 2^{n\eta \frac{s}{p}} \big| \Gg}^{1/r}\\
&\leq2^{n\kappa/2} \EE_\Gg\bra{\abs{Y_F}^{pq}}^{1/q} 2^{-n\eta \frac{s}{pr}}\EE_\Gg\bra{\abs{Y_F}^s}^{1/r}
\leq 
C2^{n\bra{\frac{\kappa}{2}-\eta\frac{s}{pr}}} \norm{H_F}^{\frac{1}{2}(p+\frac{s}{r})}.
\end{align*}
where $s\geq 1$ and $q,r> 1$ satisfy $q^{-1}+r^{-1}=1$. 
Under the assumption of $\Ee$ we have $\norm{H_F} \leq 2^{\frac{n}{2}\cdot\frac{\eta}{p}}$ (recall $F$ is of size $2^{n/2}$), and so, 
\begin{align*}
\PP\bra{Y_E \in \Omega \big| \Gg} 
&\leq C2^{n\bra{\frac{\kappa}{2}-\eta\frac{s}{pr}} +\eta\frac{n}{4}\bra{1 + \frac{s}{pr}}}
= C2^{\frac{n}{4}\bra{2\kappa+\eta\bra{1-\frac{3s}{pr}}}}.
\end{align*}
Taking $s\geq \frac{pr}{3}(1+\frac{2}{\eta}(\kappa+4p))$ ensures that $\frac{n}{4}(2\kappa+\eta(1-\frac{3s}{pr})) \leq -2pn$ and we conclude that
\begin{equation}\label{e-calculating-prob}
\PP\bra{Y_E \in \Omega \big| \Gg} \leq C2^{-2np}. 
\end{equation}
Define the event $\hat{\Ee}:=\Ee \cap \{Y_E \notin \Omega\}$. Note that under $\hat{\Ee}$, we have 
\[
\norm{H} \leq \frac{1}{2}\norm{H_F}\norm{H_E^{-1}}\norm{H_G} \leq \frac{12}{6}\bra{2^{\frac{n}{2}\cdot \frac{\eta}{p}}}^2 = C2^{n\eta/p}.
\]

Let us write $x=Y_E$ for shorthand. 
In contrast to the initial step, we cannot directly apply the polynomial perturbation result of Proposition \ref{p-poly-perturb-bound} to the conditional distribution of $Y_F$ because $g_x \sim N(Jx,H)$ is not centred. Instead, we need need to make the change of variable $y:=Jx+H^{1/2}u$, so that now we have $g_x(y)=(\textup{det}\, H)^{-1/2} \phi(u)$. 
Similarly, we define $h_x(u):= (\textup{det}\, H)^{1/2}f_x(Jx+H^{1/2}u)$. 

Then provided $\hat{\Ee}$ holds, (\ref{e-davie-28}) implies the estimate 
\begin{equation*}
\int_{\Aa} \bra{1+\abs{u}}^p \abs{h_x(u) - \seq{1+r_x\bra{Jx+H^{1/2}u}}\phi(u)}\, du \leq C_{\kappa,p}2^{n\kappa\bra{10\eta-\frac{1}{2}}},
\end{equation*}
where $\Aa=\{u\in \RR^{d_2} : \abs{Jx+H^{1/2}u} \leq 2^{n{\eta}/{p}}\}$.
If $u\notin \Aa$, that is $\abs{y} \geq 2^{n{\eta}/{p}}$, then $\abs{H^{1/2}u} \leq 2\abs{y}$ and so $\abs{u}\leq C2^{n\frac{\eta}{2p}}\abs{y}$. 
As previously argued in \cite{davie2014kmt}, using the fact that $x\notin \Omega$ under $\hat{\Ee}$ it follows that 
\[ 
\int_{\Aa^c} \bra{1+\abs{u}}^p h_x(u)\, du \leq C_{\kappa,p}2^{n\bra{\eta-\frac{\kappa}{2}}}. 
\]
Moreover, the exponential tail property of the Gaussian distribution ensures that
\[
\int_{\Aa^c} \bra{1+\abs{u}}^p \abs{1 + r_x\bra{Jx+H^{1/2}u}}\phi(u)\, du \leq C_{\kappa,p}2^{-n\kappa/2}. 
\]
Combining the last three estimates yields
\begin{equation*}\label{e-davie-31}
\int_{\RR^{d_2}} \bra{1+\abs{u}}^p \abs{h_x(u) - \seq{1+r_x\bra{Jx+H^{1/2}u}}\phi(u)}\, du \leq C_{\kappa,p}2^{n\kappa\bra{10\eta-\frac{1}{2}}}.
\end{equation*}
As in the initial case, in light of the previous estimate we apply Proposition \ref{p-poly-perturb-bound} with $\Sigma=I_{d_2}$, $\delta=C2^{n\kappa\bra{10\eta-\frac{1}{2}}}$, $K\leq C$, $n_0 = 2$, $n=\kappa$ and $\varepsilon=2^{n\bra{10\eta-\frac{1}{2}}}$ to find that, conditional on $\hat{\Ee}$,
\begin{align*}
\Ww_p\bra{h_x, \phi}
&\leq C_{\kappa,p} 2^{n\bra{21\eta-1}}.
\end{align*}
Again we have used the fact $\kappa\geq 2p$. Assuming $\hat{\Ee}$, $\norm{H^{1/2}} \leq C2^{n\frac{\eta}{2p}}$ and so 
\[
\Ww_p\bra{f_x,g_x} \leq \norm{H^{1/2}}\Ww_p\bra{h_x,\phi} \leq C_{\kappa,p}2^{n\bra{22\eta-1}}.
\] 

As Davie writes in \cite{davie2014kmt}, the situation can be summarised as follows: conditional on $Y_E=x$ and assuming $\hat{\Ee}$, we can find a random variable $Z_F^*$ with density $g_x$ such that $\abs{Z_F^* - Y_F}_{L^p} \leq C_{\kappa,p}2^{n\bra{22\eta-1}} $. If $\hat{\Ee}$ fails then we generate an independent random variable $Z^*$ with density $g_x$ and set $Z_F^*:=Z^*$. 
As in the initial case, the exponential tail bounds of Lemma \ref{l-h-exp-tails} imply that $\PP\bra{\Ee^c} \leq C2^{-2np}$ and so together with (\ref{e-calculating-prob}) we certainly have $\PP\bra{\hat{\Ee}^c} \leq C2^{-2np}$. 
Thus taking expectations over $\Gg$ and $Y_E$ and applying the tower property, we find that unconditionally, 
\begin{equation}\label{e-d1}
\EE\bra{\abs{Z_F^*-Y_F}^p} \leq C_{\kappa,p}2^{np\bra{22\eta-1}} .
\end{equation}
We can now complete the recursive step by defining
\begin{equation}\label{e-d2}
Z_F:= Z_F^* + H_F H_E^{-1} (Z_E - Y_E), 
\end{equation}
which has the correct conditional density $g_{\tilde{x}}$ with $\tilde{x}=Z_E$. Then we must have $Z_G= 2^{1/2}Z_E - Z_F$. Moreover, setting $Z_F^* := 2^{1/2}Y_E - Z_F^*$, (\ref{e-d1}) and (\ref{e-d2}) hold with $F$ replaced with $G$.

\textbf{Conclusion of the proof.}
Consider a given dyadic set $E$ of size $2^n$. We can uniquely write the expansion $E=E_k \subseteq E_{k-1} \subseteq \ldots \subseteq E_0$ where $k=m-n$ and, for each $j$, $E_j$ is a dyadic set of size $2^{m-j}$. From (\ref{e-d2}) we obtain
\[
Z_E - Y_E = \sum_{j=1}^k H_{E_k}H^{-1}_{E_j} (Z^*_{E_j} - Y_{E_j}) + H_{E_k}H_{E_0}^{-1}(Z_{E_0}-Y_{E_0}). 
\]
By the H\"{o}lder inequality, the exponential tail estimates of $\norm{H_E}$ and the fact that $\norm{H_E^{-1}}\leq 12$, for all $r \in [1,p)$ we have
\begin{align*}
\abs{{Z_E - Y_E}}_{L^r}
&\leq 12\sum^k_{j=1} \abs{\norm{H_{E_k}}}_{L^{q}} \abs{{Z^*_{E_j}-Y_{E_j}}}_{L^{p}} + 12\abs{\norm{H_{E_k}}}_{L^{q}}\abs{{Z_{E_0}-Y_{E_0}}}_{L^{p}}\\
&\leq C_{\kappa,p}\sum_{j=0}^k 2^{(m-j)\bra{22\eta-1}}
\leq C_{\kappa,p} 2^{(m-k)\bra{22\eta-1}} = C_{\kappa,p} 2^{n\bra{22\eta-1}}.
\end{align*}
Here, $q=\frac{pr}{p-r}$ (so that $q^{-1}+r^{-1}=1$). 
Taking suitably small $\eta\in (0,\frac{1}{44})$ ensures that $22\eta-1<-\frac{1}{2}$ and the proof is complete. 
\end{proof}

\section*{Acknowledgements}

The author would like to thank Prof.~Sandy Davie of Edinburgh for answering many questions about his original proofs in \cite{davie2014kmt, davie2014pathwise, davie2015poly}. 
The research is supported by the European Research Council under the European Union's Seventh Framework Programme (FP7-IDEAS-ERC, ERC grant agreement nr. 291244).

\bibliographystyle{plain}
\bibliography{hoff_library,coupling_library,homemade_library}

\end{document}